\documentclass[12pt,letterpaper]{amsart}
\usepackage{geometry}
\usepackage{mathpazo} %Fonts
\usepackage{amsmath,amssymb,amsfonts,amsthm,amscd}
\usepackage{mathrsfs,latexsym,url,setspace,xcolor,graphicx}
\newtheorem{theorem}{Theorem}
\newtheorem{theoremA}{Theorem}
\newtheorem*{theorem*}{Theorem}

\newtheorem{proposition}{Proposition}
\newtheorem{lemma}{Lemma}

\theoremstyle{remark}

\newtheorem{remark}{Remark}

\newtheorem{example}{Example}

\newtheorem{conjecture}{Conjecture}
\newtheorem{assum}{Assumption}
\newcommand{\HH}{\text{HH}}
\newcommand{\fr}{\text{Fr}}

\newcommand{\id}{\text{Id}}
\newcommand{\sym}{\text{Sym}}

\newcommand{\Aut}{\text{Aut}}
\newcommand{\spec}{\text{Spec}}

\newcommand{\gr}{\text{gr}}

\usepackage[author={S},color=yellow,opacity=0.6,markup=Highlight]{
pdfcomment}
\onehalfspace

\title{On automorphisms of enveloping algebras}

\author{Akaki Tikaradze}
\email{ tikar06@gmail.com}
\address{University of Toledo, Department of Mathematics \& Statistics, 
Toledo, OH 43606, USA}

\date{\today}
\begin{document}

\begin{abstract}

Given an algebraic Lie algebra $\mathfrak{g}$ over $\mathbb{C}$, we canonically associate to it
a Lie algebra $\mathcal{L}_{\infty}(\mathfrak{g})$ defined over $\mathbb{C}_{\infty}$--the reduction of $\mathbb{C}$ 
modulo the  infinitely
large prime, and show that  for a class of Lie algebras $\mathcal{L}_{\infty}(\mathfrak{g})$ 
is an invariant of the derived category of $\mathfrak{g}$-modules.
We give two applications of this construction. First, we show that the bounded derived category of $\mathfrak{g}$-modules
determines algebra $\mathfrak{g}$ for a class of Lie algebras. Second,
given a semi-simple Lie algebra $\mathfrak{g}$ over $\mathbb{C}$, we construct a canonical homomorphism 
from the group of automorphisms of the enveloping algebra $\mathfrak{U}\mathfrak{g}$ 
to the group of Lie algebra automorphisms of $\mathfrak{g}$,
such that its kernel does not contain a nontrivial semi-simple automorphism. 
As a corollary we obtain that any finite subgroup of automorphisms of $\mathfrak{U}\mathfrak{g}$ 
 isomorphic to a subgroup of Lie algebra automorphisms of $\mathfrak{g}.$

\end{abstract}
%%%%%%%%%%%%%%%%%%%%%%%%%
\maketitle
%%%%%%%%%%%%%%%%%%%%%%%%%

\section*{Introduction}
%Throughout given an algebra $A,$ we will denote its center by $Z(A),$ and
%its group of algebra automorphisms by $\Aut(A)$.  Also given a Lie algebra $L$, we 
%will denote its enveloping algebra by $\mathfrak{U}L.$

This paper is motivated by the question whether a Lie algebra $ \mathfrak{g}$ over $\mathbb{C}$
can be recovered from its enveloping algebra $\mathfrak{U}\mathfrak{g}.$ One way to make this question
more precise is to state the isomorphism problem for enveloping algebras: given two finite dimensional
Lie algebras $\mathfrak{g}_1,\mathfrak{g}_2$ over $\mathbb{C},$ such that their enveloping algebras are isomorphic 
$\mathfrak{U}\mathfrak{g}_1\cong \mathfrak{U}\mathfrak{g}_2$, does it follow that $\mathfrak{g}_1\cong\mathfrak{g}_2?$
This problem is widely open in general, it is known to have the positive answer for the cases of semi-simple Lie algebras (easily follows from\cite{AP}) and low dimensional nilpotent Lie algebras (see \cite{H}).

 One is tempted to
upgrade this isomorphism problem to the following derived isomorphism problem.

\begin{conjecture}
Let $\mathfrak{g}_1, \mathfrak{g}_2$ be finite dimensional Lie algebras over $\mathbb{C}.$
If the derived categories of bounded complexes of  $\mathfrak{U}\mathfrak{g_1}$-modules and $ \mathfrak{U}\mathfrak{g_2}$-modules are equivalent,
then $\mathfrak{g_1}\cong \mathfrak{g}_2.$
\end{conjecture}
We show that the conjecture holds if Lie algebras $\mathfrak{g}_1,  \mathfrak{g}_2$ satisfy
Assumption \ref{Assump} below (see Theorem \ref{Derived}.)

A closely related problem is to understand $\Aut(\mathfrak{U}\mathfrak{g})$--the group of automorphisms of the enveloping algebra.
Of particular interest are its finite subgroups.
In this regard, 
the study of finite subgroups of automorphisms of  $\mathfrak{U}\mathfrak{g}$ for semi-simple $\mathfrak{g}$
and the corresponding fixed point rings have been of great interest for some time now,
see \cite{AP}, \cite{C}, \cite{CG}, \cite{J2}. For the special case of $\mathfrak{g}=\mathfrak{sl}_2$, all finite subgroups
of $\Aut(\mathfrak{U}\mathfrak{sl}_2)$ where classified by Fleury \cite{F}. More specifically,
she proved that if $\Gamma$ is a finite subgroup of $\Aut(\mathfrak{U}\mathfrak{sl}_2)$, then $\Gamma$
is conjugate to a subgroup of $\Aut(\mathfrak{sl}_2)\subset \Aut(\mathfrak{U}\mathfrak{sl}_2).$
The proof in \cite{F} relies on the explicit knowledge of generators of automorphism groups of
primitive quotients of $\mathfrak{U}\mathfrak{sl}_2$ (a result by Dixmier), no such results
are known for higher rank Lie algebras.

Following ideas and results
of  Kontsevich and Belov-Kanel on automorphisms of the  Weyl algebra \cite{BK}, \cite{K},
we approach these problems by reducing modulo large prime $p.$
In this context it is convenient to use the reduction modulo the  infinitely large prime construction.
Recall that given a commutative ring  $R$, its reduction modulo the infinitely large 
prime $R_{\infty}$ is defined as follows (see \cite{K}, \cite{BK})
$$R_{\infty}=\varinjlim_{f.g. S\subset R}(\prod_{p\in P} S/pS)/(\bigoplus_{p\in P} S/pS),$$
here the direct limit is taken over all finitely generated subrings $S\subset R$, and $P$ denotes
the set of all prime numbers \cite{K}. We have the canonical inclusion $R\otimes \mathbb{Q}\hookrightarrow R_{\infty}.$
In particular, if $R$ is an integral domain, then $R_{\infty}$ can (and will) be viewed is
an $R$-algebra via the canonical embedding $R\hookrightarrow R_{\infty}.$
Also, we have the Frobenius map $\fr_{\infty}:R_{\infty}\to R_{\infty},$  defined as follows:
$$\fr_{\infty}(\prod_px_p)=\prod_px_p^p.$$

 All results in this paper are about Lie algebras satisfying the following assumption.
 Examples of such Lie algebras besides semi-simple and Frobenius ones are 
certain $\mathbb{Z}_2$-contractions of reductive algebras, 
for example $\mathfrak{sl}_n(\mathbb{C})\ltimes \mathbb{C}^n$ (see \cite{Pa}).

\begin{assum}\label{Assump}
Let $\mathfrak{g}$ be an algebraic Lie algebra over $\mathbb{C}$ corresponding to a connected algebraic group
$G,$ with the trivial center, such that the following properties hold. The algebra of invariants
$\sym(\mathfrak{g})^\mathfrak{g}=\mathbb{C}[f_1,\cdots,f_n]$  is a polynomial algebra with homogeneous generators 
$f_1,\cdots, f_n;$ such that they form a regular sequence in $\sym(\mathfrak{g}).$ Moreover, 
the corresponding algebra of coinvariants $A=\sym(\mathfrak{g})/(f_1,\cdots,f_n)$ is a normal domain, 
such that the coadjoint action of $G$ on $\spec A$ has an open orbit.
\end{assum}

Given a perfect Lie algebra $\mathfrak{g}$: $ [\mathfrak{g}, \mathfrak{g}]=\mathfrak{g},$
satisfying Assumption \ref{Assump},
we  construct a canonical group homomorphism (Section 4)
$$D:\Aut(\mathfrak{U}\mathfrak{g})\to \Aut(\mathfrak{g}_{\mathbb{C}_{\infty}}),$$
where $\mathfrak{g}_{\mathbb{C}_{\infty}}=\mathfrak{g}\otimes_\mathbb{C}\mathbb{C}_{\infty}.$
We have a base change homomorphism 
$\fr_{\infty}^{*}:\Aut(\mathfrak{g})\to\Aut(\mathfrak{g}_{\mathbb{C}_{\infty}})$
induces by the Frobenius embedding $\fr_{\infty}:\mathbb{C}\to\mathbb{C}_{\infty}.$
The following is the main result of the paper.
%Here $D$ is defined as the  composition of $\hat{D}$ with the restriction homomorphism 
%$\Aut(\mathfrak{U}\mathfrak{g})\to \Aut(B).$

\begin{theorem}\label{kernel}
Let $\mathfrak{g}$ be a perfect Lie algebra over $\mathbb{C}$ satisfying Assumption \ref{Assump}. 
Then there are no nontrivial semi-simple elements in $\ker(D).$
Moreover, $D$ restricts to $\mathrm{Fr}_{\infty}^{*}$ on $\mathrm{Aut}(\mathfrak{g}).$
In particular, if $\Gamma$ is a finite subgroup of $ \mathrm{Aut}(\mathfrak{U}\mathfrak{g}),$ 
then there exists  a subgroup $\Gamma'$ of $\mathrm{Aut}(\mathfrak{g}),$
such that $\Gamma$ is isomorphic to $\Gamma'$ as an abstract group.

\end{theorem}

%Put $\mathfrak{g}_{\mathbb{C}_{\infty}}=\mathfrak{g}\otimes_\mathbb{C}\mathbb{C}_{\infty}.$ 

%Denote by $g_i$ the image of $f_i$ under the symmetrization isomorphism $\sym \mathfrak{g}\to Z(\mathfrak{U}\mathfrak{g}).$
%Denote by  $B$ the quotient $\mathfrak{U}\mathfrak{g}/(g_1,\cdots, g_n),$
 %Thus, under the induced PBW grading on $B,$ we have $\gr(B)=A.$

% As a corollary we obtain the following.
%\begin{theorem}\label{duck}
%Let $\mathfrak{g}$ be a perfect Lie algebra satisfying Assumption \ref{Assump}.
%\end{theorem}

\begin{remark}
The most interesting application of the above result is for the case of a simple Lie algebra $\mathfrak{g}.$
In principle, this result provides a full classification of isomorphism classes of finite groups of automorphisms
of  $\mathfrak{U}\mathfrak{g}.$
However, although the construction of the subgroup $\Gamma'\subset \Aut(\mathfrak{g})$ is somewhat canonical,
we don not know if $\Gamma'$ is conjugate to $\Gamma$ in $\Aut(\mathfrak{U}\mathfrak{g}).$ 
In fact, we do not know if a much stronger statement about linearizability holds: Given a finite subgroup
$\Gamma\subset\Aut(\mathfrak{U}\mathfrak{g}),$
whether there exists a subgroup $\Gamma'\subset \Aut(\mathfrak{g}),$ such that $\Gamma'$ is conjugate to $\Gamma$
in $\Aut(\mathfrak{U}\mathfrak{g}).$

\end{remark}
 
 Throughout the paper, given an abelian group $M$, we  denote by $M_p$ its reduction modulo $p$: $M_p=M/pM.$

 %Let $f_1,\cdots, f_n$ be the usual generators of the center of the enveloping algebra $\mathfrak{U}\mathfrak{g}.$
%Therefore 
% $$Z(\mathfrak{U}\mathfrak{g})=\mathbb{C}[f_1,\cdots, f_n],\quad f_i\in\mathfrak{g}\mathfrak{U}\mathfrak{g}.$$ 
 %Also  $(f_1,\cdots,f_n)$ is a regular sequence in $\mathfrak{U}\mathfrak{g}$ by 
%results of Kostant.

%Recall that under the usual PBW filtration on $B,$ the associated grades algebra $\gr(B)$ is isomorphic
%to $\mathbb{C}[\mathcal{N}]$, where $\mathcal{N}$ is the nilpotent cone of $\mathfrak{g}^*.$
%Since any automorphism of $\mathfrak{U}\mathfrak{g}$ must preserve ideal $I,$

%At first, we will construct a homomorphism $\tilde{D}:\Aut(B)\to \Aut(\mathfrak{g}_{\mathbb{C}_{\infty}}),$
%and verify that its kernel contains no nontrivial semi-simple automorphisms.
%The desired homomorphism $D$ will be defined as the composition of $\tilde{D}$ and the restriction
%homomorphism $\Aut(\mathfrak{U}\mathfrak{g})\to\Aut(B).$

%\begin{proposition}\label{Main}
%Let $\Gamma\subset \Aut(B)$ be a finite subgroup of automorphisms. Then $\Gamma$ is isomorphic to a subgroup of 
%Lie algebra automorphisms of $\mathfrak{g}.$

%\end{proposition}

%For this purpose, we will fix once and for all $\mathbb{Z}[\frac{1}{l!}]$-models $\mathfrak{g}_{\mathbb{Z}}, B_{\mathbb{Z}}$
%of $\mathfrak{g}, B$ respectively, where $l$ is large enough natural number. Hence for a commutative
 %$\mathbb{Z}[\frac{1}{l!}]$-ring $S,$ we will put 
%$$\mathfrak{g}_S=\mathfrak{g}_{\mathbb{Z}}\otimes S,\quad B_S=B_{\mathbb{Z}}\otimes_{\mathbb{Z}}S.$$

\section{Reduction modulo $p$ Lemmas}

In this section, given any finite dimensional Lie algebra $\mathfrak{g}$ over $\mathbb{C},$ we 
define $\mathbb{C}_{\infty}$-Lie algebras $\mathfrak{g}_{\infty},\hat{\mathfrak{g}_{\infty}},$
and compute them for Lie algebras satisfying Assumption \ref{Assump}, Lemma \ref{center}.
This construction  plays the crucial role in proving our main results.

At first, recall that given an associative flat $\mathbb{Z}$-algebra $R$ and a prime number $p,$
 then the center $Z(R_p)$ of its reduction modulo $p$  acquires the natural Poisson bracket, which
 we  refer to as the deformation Poisson bracket,
defined as follows. Given $a, b\in Z(R_p)$, let $z, w\in R$ be their lifts respectively. 
Then the Poisson bracket $\lbrace a, b \rbrace$ is defined to be $$\frac{1}{p}[z, w] \mod p\in Z(R_p).$$

Let $S$ be a finitely generated subring of $\mathbb{C},$ and 
let $\mathfrak{g}$ be a Lie algebra over $S$ which is a finite rank free $S$-module.
Throughout the paper we denote by $B$ the quotient of $\mathfrak{U}\mathfrak{g}$ by the augmentation ideal of its center
$$B=\mathfrak{U}\mathfrak{g}/(Z(\mathfrak{U}\mathfrak{g})\cap \mathfrak{g}\mathfrak{U}\mathfrak{g})\mathfrak{U}\mathfrak{g}.$$
Then we define $S_{\infty}$-Lie algebras $\mathcal{L}_{\infty}(\mathfrak{g}), \mathcal{L}_{\infty}'(\mathfrak{g})$ as follows.
Let $p>>0$ be a  large prime number. Then the  following augmentation ideals 
$$\mathfrak{m}_p=Z(\mathfrak{U}\mathfrak{g}_p)\cap \mathfrak{g}_p(\mathfrak{U}\mathfrak{g}_p),\quad \mathfrak{n}_p=Z(B_p)\cap\mathfrak{g}_pB_p $$
 are easily seen to be  Poisson ideals in $Z(\mathfrak{U}\mathfrak{g}_p),  Z(B_p)$ respectively.
Hence $\mathfrak{m}_p/\mathfrak{m}_p^2, \mathfrak{n}_p/\mathfrak{n}_p^2$ are Lie algebras,
and we  view them as Lie algebras over $S_p$ via the the Frobenius map $\fr_p:S_p\to S_p.$
Denote $\mathfrak{m}_p/\mathfrak{m}_p^2$ by $\mathcal{L}_p(\mathfrak{g})$, and $\mathfrak{n}_p/\mathfrak{n}_p^2$ 
by $\mathcal{L}_p'(\mathfrak{g}).$ 
Similarly given a base change $S\to \bold{k}$ to a field of characteristic $p$ we define  $\bold{k}$-Lie algebras 
$\mathcal{L}_{\bold{k}}(\mathfrak{g}), \mathcal{L}_{\bold{k}}'(\mathfrak{g}).$
 We put
$$\mathcal{L}_{\infty}(\mathfrak{g})=\prod_{p\in P}\mathcal{L}_p(\mathfrak{g})/\bigoplus_{p\in P}\mathcal{L}_p(\mathfrak{g}),\quad
\mathcal{L}_{\infty}'(\mathfrak{g})=\prod_{p\in P}\mathcal{L}_p'(\mathfrak{g})/\bigoplus_{p\in P}\mathcal{L}_p'(\mathfrak{g}), $$
Now, given a Lie algebra over $\mathbb{C},$ we define 
$$\mathcal{L}_{\infty}(\mathfrak{g})=\varinjlim_{f.g. S\subset \mathbb{C}}\mathcal{L}_{\infty}(\mathfrak{g}_S),
\quad \mathcal{L}_{\infty}'(\mathfrak{g})=\varinjlim_{f.g. S\subset \mathbb{C}}\mathcal{L}_{\infty}'(\mathfrak{g}_S).$$
Here, $\mathfrak{g}_S$ is a model of $\mathfrak{g}$ over $S:$ $\mathfrak{g}_S\otimes_S\mathbb{C}=\mathfrak{g}.$
We have the natural surjective homomorphism $\mathcal{L}_{\infty}(\mathfrak{g})\to \mathcal{L}_{\infty}'(\mathfrak{g}).$
If $\mathfrak{g}_S$ is an algebraic Lie algebra over $S,$ then
we  construct below a canonical Lie algebra homomorphism 
 $$\mathfrak{g}_{S_{\infty}}=\mathfrak{g}_S\otimes_SS_{\infty}\to \mathcal{L}_{\infty}(\mathfrak{g}).$$
 It  follows from Lemma \ref{center} below that if $\mathfrak{g}$ satisfies Assumption \ref{Assump}, then
 $\mathcal{L}_{\infty}(\mathfrak{g})$ is isomorphic to a trivial central extension of
 $\mathfrak{g}_{\mathbb{C}_{\infty}}=\mathfrak{g}\otimes_{\mathbb{C}}\mathbb{C}_{\infty},$ while
 $\mathcal{L}_{\infty}'(\mathfrak{g})\cong\mathfrak{g}_{\mathbb{C}_{\infty}}.$

 Next we recall a key computation of the Poisson bracket for restricted Lie
algebras due to Kac and Radul \cite{KR}. First, we recall and fix some notations associated with enveloping algebras of restricted Lie algebras.
Let $R$ be a  commutative reduced ring of  characteristic $p>0.$
Let $\mathfrak{g}$ be a restricted Lie algebra over $R$ ($\mathfrak{g}$ is assumed to be a finite free $R$-module) 
with the restricted structure map $g\to g^{[p]}, g\in \mathfrak{g}.$
Then by $Z_p(\mathfrak{g})$ we  denote the central $R$-subalgebra of the enveloping algebra $\mathfrak{U}\mathfrak{g}$
generated by elements of the form $g^p-g^{[p]}, g\in \mathfrak{g}.$ It is well-known that the map $g\to g^{p}-g^{[p]}$ induces
homomorphism of $R$-algebras $$i:\sym (\mathfrak{g})\to Z_p(\mathfrak{g}),$$
where $Z_p(\mathfrak{g})$ is viewed as an $R$-algebra via the Frobenius map $F:R\to R.$
The homomorphism $i$ is an isomorphism when $R$ is perfect.
Also, recall that
the Lie algebra bracket on $\mathfrak{g}$ defines the Kirillov-Kostant Poisson bracket on the symmetric algebra $\sym(\mathfrak{g}).$

%As usual $W_2(\bf{k})$ denoted the ring of Witt vectors of length 2. Thus $W_2(\bf{k})$ is flat over
%$\mathbb{Z}/p^2\mathbb{Z}$ and $W_2(\bold{k})/pW_2(\bold{k})=\bf{k}.$

The following is the key  result from \cite{KR}.
We  include a proof for the reader's convenience.
\begin{lemma}\label{key}
Let $R$ be a finitely generated integral domain over $\mathbb{Z}.$
Let $G$ be an affine algebraic group over  $R$, let $\mathfrak{g}$ be its Lie
algebra. Let $G_p, \mathfrak{g}_p$ be reductions modulo $p$ of $G, \mathfrak{g}$ respectively.
Thus $Z(\mathfrak{U}\mathfrak{g}_p)$ is equipped with the deformation Poisson bracket.
Then $Z_p(\mathfrak{g}_p)$ is a Poisson subalgebra of $Z(\mathfrak{U}\mathfrak{g}_p),$ moreover the induced Poisson
bracket coincides with the negative of the Kirrilov-Kostant bracket:
$$\lbrace a^p-a^{[p]}, b^p-b^{[p]}\rbrace=-([a, b]^p-[a,b]^{[p]}),\quad a\in \mathfrak{g}_p, b\in \mathfrak{g}_p.$$

\end{lemma}
\begin{proof}
The proof directly follows from a similar result about Weyl algebras in \cite{BK}.
More specifically, let $X$ be a smooth affine variety over $R$, and let $D_X$ denote the algebra of
differential operators on $X.$ Put $\overline{X}=X\mod p$. Then the center of $D_X\mod p=D_{\overline{X}}$ can be identified with
(the Frobenius twist) of the functions on the cotangent bundle $T^{*}_{\overline{X}}$ \cite{BMR}. Then the deformation Poisson
bracket of $Z(D_{\overline{X}})$ is equal to the negative of the Poisson bracket coming from the symplectic structure
of the cotangent bundle of $\overline{X}.$

Now let $\theta:\mathfrak{g}\to D_G$ be the realization of $\mathfrak{g}$ as left-invariant vector fields
on $G.$ Then we have the corresponding embedding
$\theta:\mathfrak{U}\mathfrak{g}\to D_G$ and the corresponding reduction modulo $p $
$\bar{\theta}:\mathfrak{U}\mathfrak{g}_p\to D_{G_p}$,
which induces the embedding  $Z_p(\mathfrak{g}_p)\to Z(D_{G_p}).$
In this way $Z_p(\mathfrak{g}_p)$ is a Poisson subalgebra of $Z(D_{G_p})$ and the assertion follows. 
\end{proof}
It is clear that $i^{-1}(\mathfrak{m}_p)=\mathfrak{g}_p\sym \mathfrak{g}_p.$ So, in view of Lemma\ref{key} we have the canonical
Lie algebra homomorphism $\mathfrak{g}_p\to \mathfrak{m}_p/\mathfrak{m}_p^2.$
Hence, we obtain the desired canonical homomorphism 
$$\mathfrak{g}_{S_{\infty}}\to \mathcal{L}_{\infty}(\mathfrak{g}).$$

We have the following

\begin{lemma}\label{unique}
Let $S$ be a finitely generated subring of $\mathbb{C}.$ Let $\mathfrak{g}$ be a nilpotent Lie algebra over $S.$ Let
$p>>0$ be a prime. 
Let $I$ be a Poisson ideal of $Z(\mathfrak{U}\mathfrak{g}_p),$ such that $Z(\mathfrak{U}\mathfrak{g}_p)/I=S_p.$
Then $I/I^2\cong \mathfrak{m}_p/\mathfrak{m}_p^2$ as $S_p$-Lie algebras. 

\end{lemma}
\begin{proof}

We claim that given any ideal $I$ in $Z(\mathfrak{U}\mathfrak{g}_p),$ such that
$I\cap Z_p(\mathfrak{g}_p)=(g^p, g\in \mathfrak{g}_p)$ and $Z(\mathfrak{U}\mathfrak{g}_p)/I=S_p,$
then $I=\mathfrak{m}_p.$ 
Indeed, since $\mathfrak{U}\mathfrak{g}_p/(g^p, g\in \mathfrak{g}_p)\mathfrak{U}\mathfrak{g}_p$ is a nilpotent $S_p$-algebra,
it follows that $Z(\mathfrak{U}\mathfrak{g}_p)/(g^p, g\in \mathfrak{g}_p)Z(\mathfrak{U}\mathfrak{g}_p)$ is also a nilpotent $S_p$-algebra:
$\mathfrak{m}_p^{p^l}\subset (g^p, g\in \mathfrak{g}_p)Z(\mathfrak{U}\mathfrak{g}_p)$ for large enough $l.$
If $a+y\in I, a\in S_p, y\in \mathfrak{m}_p$, then 
$$(a+y)^{p^l}\in a^{p^l}+(g^p, g\in \mathfrak{g}_p)Z(\mathfrak{U}\mathfrak{g}_p),$$
 so $a=0$ and $I\subset\mathfrak{m}_p.$
Put $I'=I\cap Z_p(\mathfrak{g}).$ Then $I'$ is a Poisson ideal and $Z_p(\mathfrak{g})/I'=S_p.$
Hence, $I'=(g^p-\chi(g), g\in \mathfrak{g}_p)$, where $\chi:\mathfrak{g}_p\to S_p$ is a Lie algebra
homomorphism. Then $\phi(g)=g-\chi(g)$ induces an automorphism $\phi:\mathfrak{U}\mathfrak{g_p}\to\mathfrak{U}\mathfrak{g_p},$
such that $\phi(m_p)=I.$ Moreover $\phi$ admits a lift over ${S/p^2S}.$ Indeed, take any character 
$\hat{\chi}:\mathfrak{g}/p^2\mathfrak{g}\to S/p^2$ that lifts $ \chi$ and put $\tilde{\phi}(g)=\tilde{g}-\chi(g).$
Then $\tilde{\phi}:\mathfrak{U}(\mathfrak{g}/p^2\mathfrak{g})\to\mathfrak{U}(\mathfrak{g}/p^2\mathfrak{g}$) is an automorphisms
such that $\phi=\tilde{\phi}\mod p^2.$
Hence, $\phi$ is a Poisson automorphism of $Z(\mathfrak{U}\mathfrak{g}_p).$ Thus we obtain the desired isomorphism
of $S_p$-Lie algebras $I/I^2\cong \mathfrak{m}_p/\mathfrak{m}_p^2.$

%Now, let $\mathfrak{g}$ satisfies assumptions \ref{Assump}.

\end{proof}

From now on we  fix a Lie algebra $\mathfrak{g}$ satisfying Assumption \ref{Assump}.
It follows that we may choose a finitely generated subring $S\subset \mathbb{C},$
and a Lie algebra $\mathfrak{g}_S$ over $S$ which is finite free $S$-module, such that  
$\mathfrak{g}=\mathfrak{g}_S\otimes_S\mathbb{C}, f_i\in \sym \mathfrak{g}_S$ and
$A_S=\sym \mathfrak{g}_S/(f_1,\cdots,f_n)$ is a normal integral domain. Moreover
$\spec A_S$ has a nonempty open subset $U_S$ which is symplectic over $S$ under the
Kirillov-Kostant bracket.
Denote by $g_i$ the image of $f_i$ under the symmetrization isomorphism 
$\sym(\mathfrak{g})^{\mathfrak{g}}\to Z(\mathfrak{U}\mathfrak{g}).$ Hence $\gr(g_i)=f_i.$
We may assume that $S$ is large enough so that $g_i\in \mathfrak{U}\mathfrak{g}_S, 1\leq i\leq n.$
Just as before, we  put $B_S=\mathfrak{U}\mathfrak{g}_S/(g_1,\cdots, g_n).$
Hence $B=B_S\otimes_S\mathbb{C}.$ Given a commutative $S$-algebra $R$, we  denote by
$\mathfrak{g}_R, B_R$ and $A_R$ the  base changes of $\mathfrak{g}_S, B_S , A_S$
respectively. In particular, for a base change $R\to \bold{k}$, where $\bold{k}$
is an algebraically closed field, $A_{\bold{k}}$ is a normal integral domain (for $p>>0$).
We  denote images of $f_i, g_i$ in $\sym \mathfrak{g}_p, Z(\mathfrak{U}\mathfrak{g}_p)$ by
$\bar{f_i},\bar{g_i}.$ Given a commutative $S_p$-algebra $R$, we  denote by
$\mathfrak{m}_R$( respectively $\mathfrak{n}_R),$ the augmentation ideal 
$Z(\mathfrak{U}\mathfrak{g}_R)\cap \mathfrak{g}_R\mathfrak{U}\mathfrak{g}_R$ (respectively 
$Z(B_R)\cap \mathfrak{g}_RB_R.$)

%We will use the above result as follows. Let $S\subset\mathbb{C}$ be a finitely generated ring, let $p>>0$ be
%a very large prime, in particular $S/pS$ is reduced. Then $Z(B_S/pB_S)=Z(B_{S/pS})$ is equipped with the deformation
%Poisson bracket. 
%We have that $Z(B_{S/pS})=Z(B_{F_p})\otimes S/pS$ as Poisson algebras over $S/pS.$
%Let $m_S=Z(B_S/pB_S)\cap\mathfrak{g}_{S/pS}B_{S/pS}$ be the augmentation ideal of $Z(B_{S/pS}).$
%We have the following $F_p$-linear map $i_p:\mathfrak{g}_{p}\to Z(B/pB)$ defined as before: $$i_p(g)=g^p-g^{[p]}, g\in\mathfrak{g}_{F_p}.$$
%We will extand $i_p$ to $S/pS$-linear map $i_{S/pS}:\mathfrak{g}_{S/pS}\to m_S/m_S^2:$ 
%$$m_S/m_S^2=m_{p}/m_{p}^2\otimes_{F_p}S/pS,\quad i_{S/pS}=i_p\otimes_{F_p} Id_{S/pS}$$

%Then we have the following.
%\begin{lemma}\label{donkey}
%The ideal $m_S$ is the unique Poisson ideal of $Z(B_S/pB_S)$ such that $Z(B_S/pB_S)/m_S=S/pS.$ 
%The map $i_{S/pS}:\mathfrak{g}_{S/pS}\to m/m^2$ is an isomorphism of $S/pS$-Lie algebras, where the Lie
%algebra structure on $m/m^2$ is induced by the Poisson bracket on $Z(B_S/pB_S).$
%\end{lemma}
%\begin{proof}
%It suffices to prove the assertions for $S/pS=F_p.$ As it is follows from Lemma \ref{key}, $Z(B_{F_p})$ as a Poisson algebra is
%isomorphic to (the frobenius Twist of) $F_p[\mathcal{N}_{F_p}],$ where $F_p[\mathcal{N}_{F_p}]$ is equipped with the negative
%of the Kirillov-Kostant Poisson bracket. Under this isomorphism ideal $m_p\subset Z(B_{F_p})$ corresponds to
%the origin of the nilpotent cone $\mathcal{N}_{F_p}.$

%\end{proof}

Then we have the following
\begin{lemma}\label{center}
Let $\mathfrak{g}$ be as above. Then, 
$Z(\mathfrak{U}\mathfrak{g}_p)$ is a free $Z_p(\mathfrak{g}_p)$-module with the basis
  $$\lbrace\bar{g_1}^{\alpha_1}\cdots{\bar{g_n}}^{\alpha_n}, 0\leq \alpha_i<p, 1\leq i\leq n\rbrace.$$
In particular, $\mathcal{L}_{\infty}'(\mathfrak{g})=\mathfrak{g}_{S_{\infty}},$ while $\mathcal{L}_{\infty}(\mathfrak{g})$ is a trivial central
extension of $\mathfrak{g}_{S_{\infty}}.$
Moreover, given a Poisson ideal  $I$ in $Z(\mathfrak{U}\mathfrak{g}_p),$ such that $Z(\mathfrak{U}\mathfrak{g}_p)/I=S/pS,$
then $I/I^2$ is a trivial central extension of $\mathfrak{g}_p$  as an $S_p$-Lie algebra.
If $\mathfrak{g}$ is perfect, then $\mathfrak{n}_p$ is the unique Poisson ideal of
 $Z(B_p),$
such that  $Z(B_p)/\mathfrak{n}_p=S_p.$
%$$L(\mathfrak{g})=\mathfrak{g}_{S_{\infty}}\oplus \bigoplus_{i=1}^n S_{\infty}\bar{g_i},\quad L'(\mathfrak{g})=\mathfrak{g}_{S_{\infty}}.$$
%where $\bigoplus_{i=1}^n S_{\infty}\bar{g_i}$ is an abelian Lie algebra.
\end{lemma}
\begin{proof}

It is enough to verify above statements after a base change $S/pS\to \bold{k}$, where $\bold{k}$ is an algebraically closed field
of characteristic $p.$
At first, since $A_{\bold{k}}=\sym(\mathfrak{g}_{\bold{k}})/(\bar{f_1},\cdots,\bar{f_n})$ is a normal domain and the Poisson bracket is symplectic on a nonempty open subset of $\spec A_{\bold{k}},$
it follows easily that the Poisson center of $A_{\bold{k}}$ is 
$A_{\bold{k}}^p$ (see, for example, Lemma 2.4\cite{T}).
It  suffices to check that the Poisson center of $\sym\mathfrak{g}_{\bold{k}}$ is a free module
over  $(\sym\mathfrak{g}_{\bold{k}})^p$ with the basis $$\lbrace\bar{f_1}^{\alpha_1}\cdots\bar{f_n}^{\alpha_n} , 0\leq\alpha_i<p\rbrace.$$
Indeed, let $g$ be in the Poisson center of $\sym \mathfrak{g}_{\bold{k}}.$ 
Denote the ideal $(\bar{f_1},\cdots, \bar{f_n})\subset \sym\mathfrak{g}_{\bold{k}}$ by $I.$
 We claim that $g\in \sym(\mathfrak{g}_{\bold{k}})^p[\bar{f_1},\cdots, \bar{f_n}]+I^m,$ for all $m.$
 We proceed by induction on $m.$
Since $\bar{g}=g\mod I$ belongs to the Poisson center of 
$A_{\bold{k}},$ it follows that
$\bar{g}\in A_{\bold{k}}^p.$ Hence $g\in \sym(\mathfrak{g}_{\bold{k}})^p+I.$
Assume that $g\in \sym(\mathfrak{g}_{\bold{k}})^p[\bar{f_1},\cdots, \bar{f_n}]+I^m$ for some $m\geq 1.$
So, there exists $x\in g+\sym(\mathfrak{g}_{\bold{k}})^p[\bar{f_1},\cdots, \bar{f_n}],$ such that
$$x=\sum_{\alpha_1+\cdots+\alpha_n=m}\bar{f_1}^{\alpha_1}\cdots \bar{f_n}^{\alpha_n}x_{\alpha},\quad x_{\alpha}\in \sym (\mathfrak{g}_{\bold{k}}),.$$ 
Then, for any $y\in \mathfrak{g}_{\bold{k}}$, we have
$$0=\lbrace y, x\rbrace=\sum_i \bar{f_1}^{\alpha_1}\cdots \bar{f_n}^{\alpha_n}\lbrace y, x_{\alpha}\rbrace.$$
Since the sequence $(\bar{f_1},\cdots, \bar{f_n})$ is regular, we may conclude that 
$\bar{x_{\alpha}}=x_{\alpha}\mod I\in A_{\bold{k}}^p.$
Hence, $$g\in  \sym(\mathfrak{g}_{\bold{k}})^p[\bar{f_1},\cdots, \bar{f_n}]+I^{m+1}.$$
 Therefore,
$$g\in \bigcap_{m\geq1}(\sym(\mathfrak{g}_{\bold{k}})^p[\bar{f_1},\cdots, \bar{f_n}]+I^m)=\sym(\mathfrak{g}_{\bold{k}})^p[\bar{f_1},\cdots, \bar{f_n}].$$
 Now, suppose that 
 $$0= \sum_{\alpha_1,\cdots,\alpha_n}\bar{f_1}^{\alpha_1}\cdots \bar{f_n}^{\alpha_n}x_{\alpha}^p, x_{\alpha}\in \sym (\mathfrak{g}_{\bold{k}}),$$
Such that either $x_{\alpha}=0$ or $x_{\alpha}\notin I.$
Let $m$ be such that $x_{\alpha}=0$ for all $|\alpha|<m$ in the above sum. Since $I^m/I^{m+1}$
is a free $A_{\bold{k}}$-module with basis $\lbrace \bar{f_1}^{\alpha_1}\cdots\bar{f_n}^{\alpha_n}, \sum \alpha_i=m\rbrace,$
we get that $x_{\alpha}^p \mod I=0$, hence $x_{\alpha}\in I.$ So $x_{\alpha}=0.$
Thus, elements $\lbrace \bar{f_1}^{\alpha_1}\cdots\bar{f_n}^{\alpha_n} , 0\leq \alpha_i<p\rbrace$ are linearly independent over $(\sym\mathfrak{g}_{\bold{k}})^p$ 
as desired. It follows that $\mathfrak{g}_{\bold{k}}\bigoplus \oplus_{i=1}^n \bold{k}\bar{g_i}$ surjects onto
 $\mathcal{L}_{\bold{k}}(\mathfrak{g})=\mathfrak{m}_{\bold{k}}/\mathfrak{m}_{\bold{k}}^2,$
and since the center of $\mathfrak{g}_{\bold{k}}$ is trivial, we conclude that $\mathcal{L}_{\bold{k}}(\mathfrak{g})$ is a direct sum
of $\mathfrak{g}_{\bold{k}}$ with a central subalgebra spanned by images of $\bar{g}_i.$

Denote by $\tilde{f}_i$ the image of $\bar{f_i}^p$ under the isomorphism $\sym(\mathfrak{g}_{\bold{k}})\cong Z_p(\mathfrak{g}_{\bold{k}}).$
Then $\gr(\tilde{f_i})=\bar{f}_i^p.$
Since $A_{\bold{k}}$ is a domain, we have 
$$(\sym\mathfrak{g}_{\bold{k}})^p\cap I=\sum_i \bar{f_i}^p\sym\mathfrak{g}_{\bold{k}}.$$
Therefore $$Z_p(\mathfrak{g}_{\bold{k}})\cap (\bar{g_1},\cdots,\bar{g_n})=\sum \tilde{f_i}Z_p(\mathfrak{g}_{\bold{k}}).$$
 Hence, we conclude that the map $ i:\sym\mathfrak{g_{\bold{k}}}\to Z(B_{\bold{k}})$ induces an isomorphism $A_{\bold{k}}\cong  Z(B_{\bold{k}}).$
 In particular, $\mathcal{L}_{\bold{k}}'(\mathfrak{g})=\mathfrak{n}_{\bold{k}}{/\mathfrak{n}_{\bold{k}}}^2\cong\mathfrak{g}_{\bold{k}}$. 

Now let $I\subset Z(\mathfrak{U}\mathfrak{g}_p)$ be a Poisson ideal such that $Z(\mathfrak{U}\mathfrak{g}_p)/I=S_p.$
Proceeding as in the proof of Lemma \ref{unique},
without loss of generality we may assume that $(g^p-g^{[p]}, g\in \mathfrak{g}_p)\subset I.$
 Let $\bar{g_i}-a_i\in I, a_i\in S_p.$ In the above proof, replacing $\bar{g}_i$ with $\bar{g}_i-a_i,$
 we conclude that 
 $$I/I^2=\mathfrak{g}_{p}\oplus \sum_{i=1}^nS_p(\bar{g_i}-a_i).$$
If $\mathfrak{g}_p$ is perfect and
 $I$ is Poisson ideal in $B_p, $ such that $B_p/I=S/pS,$ then if follows that $g^p-g^{[p]}-\chi(g)\in I, g\in \mathfrak{g}_p,$
 such that $\chi:\mathfrak{g}_p\to S_p$ is a character. Hence, $\chi$ must be trivial.
 Since $g^p-g^{[p]}, g\in \mathfrak{g}_p$ generate $Z(B_p),$ we get that $I=\mathfrak{n}_p.$
\begin{example}
Let $\mathfrak{g}=\mathfrak{sl}_2$ with the usual generators $e, f, h.$
Then the center of $\mathfrak{U}\mathfrak{sl}_2(\mathbb{F}_p)$
is generated a $\mathbb{F}_p$-algebra by $x=e^p, y=f^p, z=h^p-h, \Delta=4fe+h^2+2h$ subject to the relation
$$(\Delta+1)^p-2(\Delta+1)^{\frac{p+1}{2}}+(\Delta+1)=4xy+z^2.$$
Thus the augmentation ideal $m_p= Z(\mathfrak{U}\mathfrak{g}_p)\cap \mathfrak{g}_p(\mathfrak{U}\mathfrak{g}_p)$
is generated by $x,y, z, \Delta$ and $m_p/m_p^2$ has an $\mathbb{F}_p$-basis
 $\lbrace\bar{x}, \bar{y}, \bar{z}, \bar{\Delta}\rbrace$--images of  $x, y, z,\Delta$ respectively. Thus we see that
$\mathcal{L}_p(\mathfrak{sl}_2)=m_p/m_p^2\cong\mathfrak{sl}_2(\mathbb{F}_p)\oplus \mathbb{F}_p\bar{\Delta}.$
Similarly $n_p$--the augmentation ideal of $Z(B_p)=Z(\mathfrak{U}\mathfrak{sl}_2(\mathbb{F}_p)/(\Delta)),$ is generated
by $x, y, z$, and $n_p/n_p^2=\mathbb{F}_p\bar{x}\oplus \mathbb{F}_p\bar{y}\oplus \mathbb{F}_p\bar{z}.$
So  $\mathcal{L}_p'(\mathfrak{sl}_2)\cong\mathfrak{sl}_2(\mathbb{F}_p).$
\end{example}

\end{proof}

\section{Derived isomorphisms}

As usual, $S$ is a finitely generated subring of $\mathbb{C}.$
Throughout, given two algebras $A, B,$ we say that they are derived equivalent if
the respective derived categories of bounded complexes of (left) modules are equivalent.
We  use the following easy consequence of \cite{R}.

\begin{lemma}\label{D}
Let $A, B$ be  flat $S$-algebras that are derived equivalent. Then $Z(A/pA)\cong Z(B/pB)$
as Poisson algebras for all $p>>0.$

% In particular, if $\mathfrak{g_1}, \mathfrak{g_2}$ are 
%Lie algebras over $\mathbb{C}$ which are either nilpotent or satisfying Assumption \ref{Assump},
% such that $\mathfrak{U}\mathfrak{g_1}, \mathfrak{U}\mathfrak{g_2}$
%are derived equivalent, then  $\mathcal{L}_{\infty}(\mathfrak{g_1})\cong  \mathcal{L}_{\infty}(\mathfrak{g_2}).$

\end{lemma}
\begin{proof}
It follows that algebras $A/p^2A, B/p^2B$ are derived equivalent.
We have the following exact sequence of $A/p^2A$-bimodules
$$0\to A/pA\to A/p^2A\to A/pA\to 0,$$
where $A/pA\to A/pA$ is the quotient map and $A/pA\to A/p^2A$ is the multiplication by $p.$
It follows from \cite{R} that the connecting map of the Hochschild cohomologies 
$$i_m:\HH^m(A/pA)\to \HH^{m+1}(A/pA)$$ corresponding to
the exact sequence above commutes with isomorphisms of Hochschild cohomologies induced by the derived
equivalence 
$$\HH^*(A/p^2A)\cong \HH^*(B/p^2B),\quad \HH^*(A/pA)\cong \HH^*(B/p^2B).$$
Now, since the deformation Poisson bracket on $Z(A/pA)$ is defined as 
$$\lbrace a, -\rbrace=i_0(a)|_{Z(A/pA)}, a\in Z(A/pA),$$
the desired result follows.

\end{proof}

Now we can easily prove the following.
\begin{theorem}\label{Derived}
Suppose that Lie algebras $\mathfrak{g}, \mathfrak{g'}$ satisfy Assumption \ref{Assump}.
 If  $\mathfrak{U}\mathfrak{g}$ is derived equivalent to  $\mathfrak{U}\mathfrak{g'}$, then $\mathfrak{g}\cong \mathfrak{g'}.$

\end{theorem}

\begin{proof}

 There exists a finitely generated subring $S$ of $\mathbb{C}$, and Lie algebras
 $\mathfrak{g}_S, \mathfrak{g'}_S$ over $S$, such that 
 $$\mathfrak{g}=\mathfrak{g}_S\otimes_S\mathbb{C},\quad \mathfrak{g'}=\mathfrak{g'}_S\otimes_S\mathbb{C},$$
and $\mathfrak{U}(\mathfrak{g}_S)$ is derived equivalent to $ \mathfrak{U}(\mathfrak{g'}_S).$
Hence, 
$$Z(\mathfrak{U}(\mathfrak{g}_{S_p}))\cong Z(\mathfrak{U}(\mathfrak{g'}_{S_p})).$$
as Poisson algebras by Lemma \ref{D}.
Let $I\subset Z(\mathfrak{U}(\mathfrak{g}_{S_p}))$ be a Poisson ideal
 such that $Z(\mathfrak{U}(\mathfrak{g}_{S_p}))/I=S_p.$
Let $I'\subset Z(\mathfrak{U}(\mathfrak{g'}_{S_p}))$ be the corresponding ideal under the above isomorphism.
Hence $I/I^2\cong I'/I'^2$ as $S_p$-Lie algebras.
 Now using Lemma \ref{center} we conclude that trivial central extensions of 
 $(\mathfrak{g})_{\mathbb{C}_\infty}$ and $ (\mathfrak{g'})_{\mathbb{C}_\infty}$
 are isomorphic. Since $\dim\mathfrak{g}=\dim\mathfrak{g'}$,
this implies that $\mathfrak{g}\cong \mathfrak{g'}.$
\end{proof}

%\begin{remark}
%In view of Lemma \ref{D}, it would be very interesting to compute $\mathcal{L}_{\infty}(\mathfrak{g})$ for nilpotent
%Lie algebras. We will come back to this question elsewhere.

%\end{remark}

\section{The homomorphisms $D, \tilde{D}$}\label{4}

 We  assume that $\mathfrak{g}$ is a perfect Lie algebra over $\mathbb{C}$ satisfying Assumption \ref{Assump}. 
Let $S$ be a large enough finitely generated subring of $\mathbb{C}$, and let $\mathfrak{g}_S$
 be a model of $\mathfrak{g}$ over $S$ (just as in the paragraph preceding Lemma \ref{center}).
Then we  construct canonical homomorphisms 
$$ D_S: \Aut(\mathfrak{U}\mathfrak{g}_S)\to \Aut(\mathfrak{g}_{S_{\infty}}),\quad \tilde{D}_S:\Aut(B_S)\to \Aut(\mathfrak{g}_{S_{\infty}}),$$
as follows. At first, remark that since $\mathfrak{g}$ is perfect, any automorphism of $\mathfrak{U}\mathfrak{g}$ must
preserve the ideal $Z(\mathfrak{U}\mathfrak{g})\cap \mathfrak{g}\mathfrak{U}\mathfrak{g}.$ Therefore we have
the restriction homomorphism $\Aut(\mathfrak{U}\mathfrak{g})\to\Aut(B).$
Let $p>>0$ be a sufficiently large prime.
Let $\phi\in \Aut(B_S).$ 
Reducing $\phi \mod p,$ we obtain $\bar{\phi}\in \Aut(Z(B_p)).$
Since $B_{p}$ is obtained by the reduction modulo $p$, we have the corresponding deformation Poisson bracket
 on its center. Hence, $\bar{\phi}$ preserves the Poisson
bracket on $Z(B_{p}).$
Now it follows from Lemma \ref{unique} that $\bar{\phi}$ preserves $\mathfrak{n}_p,$ thus it induces a Lie algebra automorphism on 
$\mathfrak{n}_p/\mathfrak{n}_p^2\cong \mathfrak{g}_p,$ which we  denote by  $(\tilde{D}_{S_p})(\phi).$
Hence, we obtain a canonical homomorphism $\tilde{D}_{S_p}:\Aut(B_S)\to\Aut(\mathfrak{g}_p).$
Also, given a base change $S_p\to\bold{k}$, we  denote the corresponding homomorphism 
$\tilde{D}_{S_p}\otimes_{S_p}\bold{k}:\Aut(B_S)\to\Aut(\mathfrak{g}_{\bold{k}})$
by $\tilde{D}_{\bold{k}}.$ Also, denote by $D_{\bold{k}}:\Aut(\mathfrak{U}\mathfrak{g}_S)\to\Aut(\mathfrak{g}_{\bold{k}})$
the composition of $\tilde{D}_{\bold{k}}$ with the restriction $\Aut(\mathfrak{U}\mathfrak{g}_S)\to \Aut(B_S).$
The element $\prod_p (\tilde{D}_{S_p})(\phi)$  gives rise
to an element of $\Aut(\mathfrak{g}_{S_{\infty}})$,
which we denote by $\tilde{D}_S(\phi).$ This way we obtain the desired homomorphisms
$\tilde{D}_S.$
We define $D_S:\Aut(\mathfrak{U}\mathfrak{g}_S)\to \Aut(\mathfrak{g}_{S_{\infty}})$ as the composition
of $\tilde{D_S}$ with the restriction $\Aut(\mathfrak{U}\mathfrak{g}_S)\to \Aut(B_S).$

Now, taking the direct limit of the above homomorphisms $\tilde{D}_S, D_S$ over finitely generated subsrings $S\subset\mathbb{C},$
  we obtain the sought after homomorphisms
$$\tilde{D}:\Aut(B)\to \Aut(\mathfrak{g}_{\mathbb{C}_{\infty}}),
\quad D:\Aut(\mathfrak{U}\mathfrak{g})\to \Aut(\mathfrak{g}_{\mathbb{C}_{\infty}}).$$
It is clear from the construction and Lemma\ref{center} that $D(\Aut(\mathfrak{g}))=\fr_{\infty}^{*},$
 where $\fr_{\infty}:\mathbb{C}\to \mathbb{C}_{\infty}$ is the canonical inclusion
followed by the Frobenius map.

%We will denote by $D^S(\phi)$ the image of $\prod_p \phi_p$ under the homomorphism 
%$$\prod_{p}\Aut(\mathfrak{g}_{S/pS}) \to\Aut(\mathfrak{g}_{S_{\infty}}).$$
%Hence we obtain a homomorphism 
%$$D_S:Aut(B_S)\to \Aut(\mathfrak{g}_{S_{\infty}}).$$
%It is clear that 
%Passing to the direct limit over $S\subset \mathbb{C}$, we obtain the desired homomorphism
%$\tilde{D}:\Aut(B)\to \Aut(\mathfrak{g}_{\mathbb{C}_{\infty}}).$

%Recall that $Z_p(\mathfrak{g})$ surjects onto $Z(B_{\bf{k}}),$ and this way $Z(B_{\bf{k}})$
%may be identified with (Frobenius twist of)  $S/pS[\mathcal{N}_p],$ where $\mathcal{N}_p$ is the nilpotent cone of $\mathfrak{g}_p.$ 
%$$\tilde{D}^S_p(\phi)=i_{S/pS}^{-1}\circ\bar{\phi}|_{m/m^2}\circ i_{S/pS}\in Aut(\mathfrak{g}_{S/pS}).$$

\section{The proof}

%The following is the crucial result.

%\begin{proposition}\label{Main}
%Let $\Gamma\subset \Aut(B)$ be a finite subgroup of automorphisms. Then $\Gamma$ is isomorphic to a subgroup of 
%Lie algebra automorphisms of $\mathfrak{g}.$

%\end{proposition}

We  start by the following.
\begin{proposition}\label{kernel*}
Let $\mathfrak{g}$ be a perfect Lie algebra satisfying Assumption \ref{Assump}.
Then the kernel of the restriction homomorphism $\mathrm{Aut}(\mathfrak{U}\mathfrak{g})\to \mathrm{Aut}(B)$ contains
no nontrivial semi-simple automorphisms.
\end{proposition}
We remark that in general, the  homomorphism $\Aut(\mathfrak{U}\mathfrak{g})\to \Aut(B)$
  is not injective.
For example, in the case of  $\mathfrak{g}=\mathfrak{sl}_2,$ this follows from existence 
of a non tame automorphism of $\mathfrak{U}\mathfrak{sl}_2,$ proved by Joseph \cite{J1}. 

For the proof of  Lemma\ref{kernel*}, we need a specific set of linearly independent elements of $\HH^2(B)$.
 Recall that for a semi-simple $\mathfrak{g}$, one has  $\dim \HH^2(B)=n,$ as follows immediately from Soergel's result \cite{S}.

Recall that $Z(\mathfrak{U}\mathfrak{g})=\mathbb{C}[g_1,\cdots, g_n], g_i\in\mathfrak{g}\mathfrak{U}\mathfrak{g}$ and $(g_1,\cdots,g_n)$ is a regular
sequence in $\mathfrak{U}\mathfrak{g}.$
Put $I=(g_1,\cdots,g_n).$ So $B=\mathfrak{U}\mathfrak{g}/I$. Let us put  $B'=\mathfrak{U}\mathfrak{g}/I^{2}$. 
Then we have a short exact sequence $$0\to I/I^2\to B'\to B\to 0.$$ 
We have $I/I^2=\bigoplus_{i=1}^n B\bar{g_i},$ where $\bar{g_i}$ denotes the image of $g_i$ under the quotient map $\mathfrak{U}\mathfrak{g}\to B'.$
Denote by $$\omega\in \HH^2(B, I/I^2)=\bigoplus_{i=1}^n \HH^2(B)\bar{g_i}$$ the cohomology class  corresponding to the above short exact sequence.
Let us put $\omega=\sum_{i=1}^n \omega_i\bar{g_i}$ where $\omega_i\in \HH^2(B).$
Under these notations we have the following.
\begin{lemma}\label{independent}
Elements $\omega_i, 1\leq i\leq n$ are linearly independent . 
\end{lemma}
\begin{proof}
Let $\sum_{i=1}^n c_i\omega_i=0, c_i\in \mathbb{C}.$ Let $c_j\neq 0$ for some $j.$
Let $B_1$ be the quotient of $B'$ by  the ideal $\sum_{i\neq j}B'\bar{g_i}.$
Then it follows that the quotient map $B_1\to B_1/B_1\bar{g_j}=B$ admits a $\mathbb{C}$-algebra splitting.
Let $\psi:B\to B_1$ be such a splitting. Let us write $$\psi(x)=x+\theta(x)\bar{g_j},\quad x\in \mathfrak{g},\quad \theta(x)\in B.$$
This implies that $\bar{g_j}\in \bar{g_j}(\mathfrak{g}B').$
 Therefore $$g_j\in \sum_{i\neq j}g_i\mathfrak{U}\mathfrak{g}+g_j\mathfrak{g}\mathfrak{U}\mathfrak{g}.$$
Thus, there exists $\alpha\in \mathfrak{U}\mathfrak{g}\setminus  \mathfrak{g}\mathfrak{U}\mathfrak{g}$ such that
$\alpha g_j\in \sum_{i\neq j}g_i\mathfrak{U}\mathfrak{g}.$ Now the regularity of the sequence $(g_1,\cdots, g_n)$ implies that
$\alpha\in I$, which is a contradiction.
 \end{proof}

\begin{proof}[Proof of  Proposition\ref{kernel*}]
Let $\phi$ be a semi-simple automorphism of $\mathfrak{U}\mathfrak{g}$ that restricts to the identity on $B.$
Denote by $\tilde{\phi}$ the restriction of $\phi$ on $B'.$
Therefore $\tilde{\phi}$ fixes $\omega=\sum \omega_i\bar{g_i}\in H^2(B, I/I^2).$
Let $\tilde{\phi}:I/I^2\to I/I^2$ be represented by the matrix 
$$(\phi_{ij}), \phi_{ij}\in \mathbb{C}, \quad \tilde{\phi}(\bar{g_i})=\sum \phi_{ji}\bar{g_j}.$$
Entries $\phi_{ij}$ are scalars since $\phi$  preserves $\mathbb{C}[g_1,\cdots,g_n].$
Thus $\sum \phi_{ij}\omega_j=\omega_i.$
Since $\omega_j$ are linearly independent by Lemma \ref{independent}, we conclude that $\phi|_{I/I^2}=\id.$ Hence $\phi|_{I^n/I^{n+1}}=\id$ for all $n.$
Since $\phi$ is semi-simple, we get that $\phi=\id.$ 
\end{proof}

\begin{remark}
It was proved by Polo\cite{P}
that when $\mathfrak{g}$ is semi-simple, the action of an automorphism of $\mathfrak{U}\mathfrak{g}$ on its center is given by a Dynkin diagram automorphism
of $\mathfrak{g}.$

\begin{proof}[Proof of Theorem \ref{kernel}.]

In view of Proposition \ref{kernel*}, it suffices to check that $\ker(\tilde{D})$ has no nontrivial semi-simple automorphisms.
Assume that $\phi\in\Aut(B)$ is a non-trivial semi-simple automorphism.
Therefore there exists a finitely generated subring $S\subset\mathbb{C},$ 
and a finite free $\phi$-invariant $S$-submodule $V\subset B_S,$
such that $\phi|_{V}$ is semi-simple over $S$ and $V$ generates $B_S$ as an $S$-algebra.
(we are using notations from the paragraph preceding Lemma \ref{center}).
We  show that for sufficiently large $S,$ and for all $p>>0,$
given any homomorphism $S_p\to \bold{k}$, where $\bold{k}$ is a field,
then $\tilde{D}_{\bold{k}}(\phi)\neq \id_{\mathfrak{g}_{\bold{k}}}.$
Let $y\in V, 1\neq\alpha\in S, $ such that $\phi(y)=\alpha y.$ 
Let us write $y=\sum a_ie_i, 0\neq a_i\in S,$ where $e_i$ are basis elements of $B_S$ as a free $S$-module.
We may assume that $a_i, 1-\alpha$ are invertible is $S.$ Now let $p>>0$ and
$\rho:S\to \bold{k}$ be a base change to a field such that 
$\tilde{D}_{\bold{k}}=\id_{\mathfrak{g}_{\bold{k}}}.$
Thus, we have a non-trivial semi-simple automorphism $\bar{\phi}\in \Aut(B_{\bold{k}}),$ such that
$\bar{\phi|}_{m_{\bold{k}}/m_{\bold{k}}^2}=\id.$
 Then $\bar{\phi}$ acts trivially on $m_{\bold{k}}^n/m_{\bold{k}}^{n+1}$ for all $n.$
Since the action of $\phi$ on $Z(B_{\bf{k}})$ is semi-simple, it follows that the action of $\bar{\phi}$ on $Z(B_{\bf{k}})$ is trivial.
Then by the Noether-Skolem theorem there exists $a\in B_{\bf{k}},$ such that
$$\phi(x)=axa^{-1}, x\in B_{\bf{k}}.$$ 
But, since $\bar{\phi}(\bar{y})=\bar{\alpha}\bar{y}$ and $1\neq \bar{\alpha}\in \bold{k}, 0\neq \bar{y}\in B_{\bold{k}}$, we 
get that $a\bar{y}=\bar{\alpha} \bar{y}a.$ Recall that under the PBW filtration on $B_{\bf{k}}$,
$\gr(B_{\bold{k}})=A_{\bold{k}}$ is a commutative domain. Hence, $$0\neq \gr(a)\gr(\bar{y})=\bar{\alpha} \gr(\bar{y})\gr(a),$$
 which is a  contradiction.

Let $\theta:\mathbb{C}_{\infty}\to\mathbb{C}$ be a homomorphism.
Then we define a (non-canonical) homomorphism
$D^*:\Aut(\mathfrak{U}\mathfrak{g})\to\Aut(\mathfrak{g})$ as the composition of $D$ with the base change homomorphisms
$\theta^*:\Aut(\mathfrak{g}_{\mathbb{C}_{\infty}})\to\Aut(\mathfrak{g}).$
Next we  show that $\ker(D^*)$ contains no nontrivial finite order elements.
This implies that given a finite subgroup $\Gamma\subset\Aut(\mathfrak{U}\mathfrak{g}),$
then $\Gamma'=D^*(\Gamma)\subset \Aut(\mathfrak{g}), \Gamma\cong \Gamma'.$
Let $\phi\in\Aut(\mathfrak{U}\mathfrak{g})$ such that $\phi^m=1, \phi\neq 1.$ We may choose a finitely generated
subring $S\subset \mathbb{C}, \frac{1}{m}\in S$ containing all $m$-th roots of unity, 
such that $\phi\subset \Aut(\mathfrak{U}\mathfrak{g}_S).$
As it was shown in the preceding paragraph, by enlarging $S$ if necessary, for all $p>>0$ and
a base change $S_p\to\bold{k}$ to a field, we have
$D_{\bold{k}}(\phi)\neq \id_{\mathfrak{g}_{\bold{k}}}.$
Let $f_p(t)\in S_p[t]$ denote the characteristic polynomial of $D_{S_p}(\phi)\in \Aut(\mathfrak{g}_{S_p}).$
Put $l=\dim\mathfrak{g}.$
We  show that
$\theta(\prod_pf_p(t))\neq (t-1)^l$ in $S_{\infty}[t].$ 
Indeed, let us write $S_p=\prod_i S_{p,i}$, where 
each $S_{p,i}$ is a domain (since $p$ is unramified in $S).$
Since the image of $D_{S_p}(\phi)$ in $\Aut(\mathfrak{g}_{S_{p,i}})$ has order $ m$,
 it follows that
the image of $f_p(t)$ in  $ S_{p,i}[t]$ is not equal to $(t-1)^l$ and is of the form $\prod_{j=1}^m(t-a_j), a_j^m=1.$
Denote by $\Psi\subset S[t]$ the finite set of all degree $l$ monic polynomials not equal to $(t-1)^l,$ 
of the form $\prod_{j=1}^m(t-a_j),a_j\in S, a_j^m=1.$
For each $g\in\Psi,$ denote by $I_g$ the set of pairs $(p,i)$ for which $f_p(t)=g $ in $S_{p,i}.$
Then we have $\prod_p S_p=\prod_{g\in \Psi}(\prod_{(p,i)\in I_g} S_{p,i}).$
Now, suppose $g\in \Psi$ is such that $\prod_{(p,i)\in I_g} S_{p,i}\not\subset \ker(\theta).$
Then it follows that $\theta(\prod_pf_p(t))=\theta(g)\neq (t-1)^l.$
Hence $\phi\notin \ker(D^*)$ as desired.

\end{proof}

\end{remark}
\textbf{Acknowledgement:} I am very grateful to the anonymous referee for the numerous useful suggestions.

\end{document}